\documentclass{amsart}

\usepackage{amsthm} 
\usepackage{amssymb} 
\usepackage[all]{xy}
\SelectTips{cm}{}

\usepackage{enumerate} 
\usepackage{hyperref}

\numberwithin{equation}{section}

\theoremstyle{plain}
\newtheorem{lemma}{Lemma}[section]
\newtheorem{theorem}[lemma]{Theorem}
\newtheorem{proposition}[lemma]{Proposition}
\newtheorem{corollary}[lemma]{Corollary}

\theoremstyle{definition} 

\newtheorem{example}[lemma]{Example}

\theoremstyle{remark}
\newtheorem{remark}[lemma]{Remark} 

\def\urltilda{\kern -.15em\lower .7ex\hbox{\~{}}\kern .04em}

\renewcommand{\dim}{\operatorname{dim}}
\newcommand{\depth}{\operatorname{depth}}
\newcommand{\Ext}{\operatorname{Ext}}
\newcommand{\hh}{\operatorname{H}} 
\newcommand{\zz}{\operatorname{Z}}
\newcommand{\bb}{\operatorname{B}}
\newcommand{\Hom}{\operatorname{Hom}}
\newcommand{\injdim}{\operatorname{inj\,dim}}
\newcommand{\projdim}{\operatorname{proj\,dim}}
\newcommand{\flatdim}{\operatorname{flat\,dim}}
\renewcommand{\le}{\leqslant}
\renewcommand{\ge}{\geqslant}

\newcommand{\RHom}{\operatorname{\mathsf{R}Hom}}
\newcommand{\Spec}{\operatorname{Spec}}
\newcommand{\Tor}{\operatorname{Tor}}
\newcommand{\splf}{\operatorname{splf}}
\newcommand{\width}{\operatorname{width}}
\newcommand{\Ltensor}{\otimes^{\mathsf{L}}}
\newcommand{\susp}{\mathsf{\Sigma}}
\newcommand{\RGamma}{\mathsf{R}\varGamma}
\newcommand{\LLambda}{\mathsf{L}\varLambda}

\newcommand{\xra}{\xrightarrow}
\newcommand{\lra}{\longrightarrow}

\newcommand{\bbZ}{\mathbb Z} 
\newcommand{\bbN}{\mathbb N}

\newcommand{\bsx}{\boldsymbol{x}}

\newcommand{\fa}{\mathfrak{a}} 
\newcommand{\fb}{\mathfrak{b}}
\newcommand{\fm}{\mathfrak{m}} 
\newcommand{\fp}{\mathfrak{p}}
\newcommand{\fn}{\mathfrak{n}} 
\newcommand{\fq}{\mathfrak{q}}

\begin{document}

\title[Rigidity of Ext and Tor]{Rigidity of Ext and Tor with
  coefficients in residue fields of a commutative noetherian ring}

\author[L.\,W.\ Christensen]{Lars Winther Christensen}

\address{L.W.C. Texas Tech University, Lubbock, TX 79409, U.S.A.}

\email{lars.w.christensen@ttu.edu}

\urladdr{http://www.math.ttu.edu/~lchriste}

\author[S.\,B.\ Iyengar]{Srikanth B. Iyengar}

\address{S.B.I. University of Utah, Salt Lake City, UT 84112, U.S.A.}

\email{iyengar@math.utah.edu}

\urladdr{http://www.math.utah.edu/~iyengar}

\author[T.\ Marley]{Thomas Marley}

\address{T.M. University of Nebraska-Lincoln, Lincoln, NE 68588,
  U.S.A.}

\email{tmarley1@unl.edu}

\urladdr{http://www.math.unl.edu/~tmarley}

\thanks{L.W.C.\ was partly supported by NSA grant H98230-14-0140 and
  Simons Foundation collaboration grant 428308; S.B.I.\ was partly
  supported by NSF grant DMS-1503044.}

\date{27 July 2017}

\keywords{Ext, Tor, rigidity, flat dimension, injective dimension.}

\subjclass[2010]{13D07; 13D05}

\begin{abstract}
  Let $\fp$ be a prime ideal in a commutative noetherian ring $R$.  It
  is proved that if an $R$-module $M$ satisfies
  $\Tor_{n}^{R}(k(\fp),M)=0$ for some $n\ge \dim R_{\fp}$, where
  $k(\fp)$ is the residue field at $\fp$, then
  $\Tor^{R}_{i}(k(\fp),M)=0$ holds for all $i\ge n$. Similar rigidity
  results concerning $\Ext_R^{*}(k(\fp),M)$ are proved, and
  applications to the theory of homological dimensions are explored.
\end{abstract}

\maketitle

\section{Introduction} 

Flatness and injectivity of modules over a commutative ring $R$ are
characterized by vanishing of (co)homological functors and such
vanishing can be verified by testing on cyclic $R$-modules. We discuss
the flat case first, and in mildly greater generality: For any
$R$-module $M$ and integer $n\ge 0$, one has $\flatdim_RM < n$ if and
only if $\Tor^R_{n}(R/\fa,M)=0$ holds for all ideals $\fa \subseteq
R$. When $R$ is noetherian, and in this paper we assume that it is, it
suffices to test on modules $R/\fp$ where $\fp$ varies over the prime
ideals in $R$.

If $R$ is local with unique maximal ideal $\fm$, and $M$ is finitely
generated, then it is sufficient to consider one cyclic module, namely
the residue field $k:=R/\fm$. Even if $R$ is not local and $M$ is not
finitely generated, finiteness of $\flatdim_R M$ is characterized by
vanishing of Tor with coefficients in fields, the residue fields
$k(\fp) := R_\fp/\fp R_\fp$ to be specific. While vanishing of
$\Tor^R_*(k(\fp),M)$ for any one particular residue field does not
imply that $\flatdim_RM$ is finite, one may still ask if vanishing of
a single group $\Tor_n^R(k(\fp),M)$ implies vanishing of all higher
groups, a phenomenon known as \emph{rigidity}.  While this does not
hold in general (cf. Example~\ref{exa:injloc}), we prove that it does
hold if $n$ is sufficiently large; see Theorem~\ref{thm:fdim} for the
proof.

\begin{theorem}
  \label{thm:1}
  Let $\fp$ be a prime ideal in a commutative noetherian ring $R$ and
  let $M$ be an $R$-module.  If one has $\Tor_{n}^{R}(k(\fp),M)=0$ for
  some integer $n \ge \dim R_\fp$, then $\Tor_i^R(k(\fp),M) = 0$ holds
  for all $i \ge n$.
\end{theorem}

\noindent As $\flatdim_RM < n$ holds if and only if one has
$\Tor_{i}^{R}(k(\fp),M)=0$ for all primes~$\fp$ and all $i\ge n$,
Theorem~\ref{thm:1} provides for an improvement of existing
characterizations of modules of finite flat dimension.

In parallel to the flat case, the injective dimension of an $R$-module
$M$ is less than~$n$ if $\Ext^{n}_R(R/\fp,M)=0$ for every prime ideal
$\fp$. Moreover the injective dimension can be detected by vanishing
\emph{locally} of cohomology with coefficients in residue fields. That
is, $\injdim_RM < n$ holds if and only if one has
$\Ext^{n}_{R_\fp}(k(\fp),M_\fp)=0$ for all primes~$\fp$. By the
standard isomorphisms $\Tor^R_*(k(\fp),M) \cong
\Tor^{R_\fp}_*(k(\fp),M_\fp)$ there is no local/global distinction for
Tor vanishing. The following consequence of
Proposition~\ref{prp:bass-finite} is, therefore, a perfect parallel to
Theorem~\ref{thm:1}.

\begin{theorem}
  \label{thm:2}
  Let $\fp$ be a prime ideal in a commutative noetherian ring $R$ and
  let $M$ be an $R$-module.  If one has
  $\Ext^{n}_{R_\fp}(k(\fp),M_\fp)=0$ for some integer $n \ge \dim
  R_\fp$, then $\Ext^i_{R_\fp}(k(\fp),M_\fp) = 0$ holds for all $i \ge
  n$.
\end{theorem}
In contrast to the situation for Tor, the cohomology groups
$\Ext_{R_\fp}^*(k(\fp),M_\fp)$ and $\Ext_R^*(k(\fp),M)$ can be quite
different, and it was only proved recently, in
\cite[Theorem~1.1]{LWCSBI}, that the injective dimension of an
$R$-module can be detected by vanishing \emph{globally} of cohomology
with coefficients in residue fields. That is, $\injdim_RM < n$ holds
if and only if one has $\Ext^{i}_{R}(k(\fp),M)=0$ for all $i\ge n$ and
all primes~$\fp$. One advantage of this global vanishing criterion is
that it also applies to complexes of modules; per
Example~\ref{exa:injdimloc} the local vanishing criterion does
not. For the proof of the following rigidity result for
$\Ext_R^*(k(\fp),M)$, see Remark~\ref{rmk:injdim}.

\begin{theorem}
  \label{thm:3}
  Let $\fp$ be a prime ideal in a commutative noetherian ring $R$ and
  let $M$ be an $R$-module.  If one has $\Ext^{n}_{R}(k(\fp),M)=0$ for
  some integer $n \ge 2\dim R$, then $\Ext^{i}_{R}(k(\fp),M)=0$ holds
  for all $i \ge n$.
\end{theorem}

The case when $\fp$ is the maximal ideal of a local ring merits
comment, for the bound on $n$ in Theorems \ref{thm:2} and \ref{thm:3}
differs by a factor of~$2$. The proof shows that it is sufficient to
require $n \ge \dim R_\fp + \projdim_R R_\fp$ in Theorem~\ref{thm:3},
and that aligns the two bounds in this special case.  For a general
prime $\fp$ however the number $\projdim R_\fp$ may depend on the
Continuum Hypothesis; see Osofsky \cite{BLO68}.

In this introduction, we have focused on results that deal with
rigidity of the Tor and Ext functors. In the text, we also establish
results that track where vanishing of these functors starts, when
indeed they vanish eventually.

\begin{center}
  $\ast \ \ \ast \ \ \ast$
\end{center}
\noindent
Throughout $R$ will be a commutative noetherian ring.  Background
material on homological invariants and local (co)homology is recalled
in Section~\ref{sec:prelim}. Rigidity results for Ext and Tor over
local rings are proved in Section~\ref{sec:local}, and applications to
homological dimensions follow in
Sections~\ref{sec:flatdim}--\ref{sec:injdim}. The final section
explores, by way of examples, the complicated nature of injective
dimension of unbounded complexes.

\section{Local homology and local cohomology} 
\label{sec:prelim}

Our standard reference for definitions and constructions involving
complexes is \cite{LLAHBF91}.  We will be dealing with graded modules
whose natural grading is the upper one and also those whose natural
grading is the lower one. Therefore, we set
\begin{equation*}
  \inf \hh_{*}(M) := \inf\{i\mid \hh_{i}(M)\ne 0\} \quad\text{and}\quad 
  \inf \hh^{*}(M) := \inf\{i\mid \hh^{i}(M)\ne 0\}
\end{equation*}
for an $R$-complex $M$, and analogously we define $\sup \hh_{*}(M)$
and $\sup\hh^{*}(M)$. We often work in the derived category of
$R$-modules, and write $\simeq$ for isomorphisms there. A morphism
between complexes is a \emph{quasi-isomorphism} if it is an
isomorphism in homology; that is to say, if it becomes an isomorphism
in the derived category.

The next paragraphs summarize the definitions and basic results on
local cohomology and local homology, following \cite{AJL-97,JLi02}.

\subsection*{Local (co)homology}
\label{ss:lc}

Let $R$ be a commutative noetherian ring and $\fa$ an ideal in
$R$. The right derived functor of the $\fa$-torsion functor
$\varGamma_\fa$ is denoted $\RGamma_\fa$, and the \emph{local
  cohomology} supported on $\fa$ of an $R$-complex $M$ is the graded
module
\begin{equation*}
  \hh^*_\fa(M) := \hh^*(\RGamma_\fa(M))\,.
\end{equation*}
There is a natural morphism $\RGamma_\fa(M)\to M$ in the derived
category; $M$ is said to be \emph{derived $\fa$-torsion} when this map
is an isomorphism. This is equivalent to the condition that
$\hh^{*}(M)$ is degreewise $\fa$-torsion; see
\cite[Proposition~6.12]{DG}.

The left derived functor of the $\fa$-adic completion functor
$\varLambda^\fa$ is denoted $\LLambda^\fa$ and the \emph{local
  homology} of $M$ supported on $\fa$ is the graded module
\begin{equation*}
  \hh_*^\fa(M) := \hh_*(\LLambda^\fa(M))\,.
\end{equation*}
There is a natural morphism $M\to \LLambda^{\fa}(M)$ in the derived
category and we say $M$ is \emph{derived $\fa$-complete} when this map
is an isomorphism. This is equivalent to the condition that for each
$i$, the natural map $\hh^{i}(M)\to \hh_{0}^{\fa}(\hh^{i}(M))$ is an
isomorphism; see \cite[Proposition~6.15]{DG}.

The morphisms $\RGamma_\fa(M)\to M$ and $M\to \LLambda^{\fa}(M)$
induce isomorphisms
\begin{equation}
  \label{eq:lc-Tor}
  \begin{aligned}
    \Ext^*_R(R/\fa,M) &\cong \Ext^*_R(R/\fa,\RGamma_\fa M)
    \quad\text{and}\\ \Tor^R_{*}(R/\fa,M) &\cong
    \Tor^{R}_{*}(R/\fa,\LLambda^\fa M)\,.
  \end{aligned}
\end{equation}

Indeed, the first one holds because the functor $\RGamma_{\fa}$ is
right adjoint to the inclusion of the $\fa$-torsion complexes (that is
to say, complexes whose cohomology is $\fa$-torsion) into the derived
category of $R$; see \cite[Proposition 3.2.2]{JLi02}. Thus one has
\begin{equation}
  \label{eq:k}
  \RHom(R/\fa,\RGamma_{\fa} M)  \xra{\ \simeq\ } \RHom(R/\fa,M)\:,
\end{equation}
and this gives the first isomorphism.  As to the second isomorphism,
consider the commutative square in the derived category
\begin{equation*}
  \xymatrix{
    \RGamma_{\fa} M \ar@{->}[r] \ar@{->}[d]_{\simeq}  & M \ar@{->}[d] \\
    \RGamma_{\fa} \LLambda^{\fa} M\ar@{->}[r] & \LLambda^{\fa} M }
\end{equation*}
induced by the vertical morphism on right. The isomorphism on the left
is part of Greenlees--May duality; see, for example, \cite[Corollary
(5.1.1)]{AJL-97}.  Applying the functor $R/\fa \Ltensor_{R}(-)$ yields
the commutative square
\begin{equation*}
  \xymatrix{
    R/\fa \Ltensor_{R} \RGamma_{\fa} M \ar@{->}[r]^{\simeq} \ar@{->}[d]_{\simeq}  
    & R/\fa \Ltensor_{R} M \ar@{->}[d] \\
    R/\fa \Ltensor_{R} \RGamma_{\fa} \LLambda^{\fa}
    M\ar@{->}[r]_{\simeq} & R/\fa \Ltensor_{R} \LLambda^{\fa} M }
\end{equation*}
The horizontal maps are isomorphisms by \eqref{eq:k} and
\cite[Proposition 6.5]{DG}. Thus the vertical map on the right is also
an isomorphism; in homology, this is the desired isomorphism.

Let $C(\fa)$ denote the \v{C}ech complex on a set of elements that
generate $\fa$. The values of the functors $\LLambda^\fa$ and
$\RGamma_\fa$ on an $R$-complex $M$ can then be computed as
\begin{equation*}
  \LLambda^\fa(M) = \RHom_R(C(\fa),M)\quad\text{and}\quad
  \RGamma_\fa(M) = C(\fa) \Ltensor_R M\,.
\end{equation*}
See for example \cite[Theorem (0.3) and Lemma (3.1.1)]{AJL-97}.

\subsection*{Depth and width} 
\label{ss:dw}

In the remainder of this section $(R, \fm, k)$ will be a local
ring. This means that $R$ is a commutative noetherian ring with unique
maximal ideal $\fm$ and residue field $k := R/\fm$.

The \emph{depth} and \emph{width} of an $R$-complex $M$ are defined as
follows:
\begin{equation*}
  \depth_R M = \inf \Ext^*_R(k,M) \quad\text{and}\quad  
  \width_R M = \inf \Tor^R_*(k,M)\,.
\end{equation*}
One has $\depth_RM \ge \inf\hh^*(M)$ and if $i =\inf\hh^*(M)$ is
finite, then equality holds if and only if $\Hom_R(k,\hh^i(M))\ne
0$. Similarly, one has $\width_RM \ge \inf\hh_*(M)$ and if
$j=\inf\hh_*(M)$ is finite, then equality holds if and only if
$k\otimes_R \hh_j(M)\ne 0$.

If $\flatdim_R M$ is finite, then one has an equality
\begin{equation}
  \label{eq:AB-formula}
  \depth_{R}M = \depth R - \sup \Tor^{R}_{*}(k,M)\,.
\end{equation}
This is an immediate consequence of \cite[Lemma 4.4(F)]{LLAHBF91}. For
finitely generated modules it is the Auslander--Buchsbaum Formula.

Similarly, if $\injdim_R M$ is finite, then one has
\begin{equation}
  \label{eq:Bass-formula}
  \width_{R}M = \depth R - \sup \Ext_{R}^{*}(k,M)\,.
\end{equation}
This is a consequence of \cite[Lemma 4.4(I)]{LLAHBF91}. For finitely
generated modules the equality above yields Bass' formula $\injdim_RM
= \depth R$.

From \cite[Definitions~2.3 and 4.3]{HBFSIn03} one gets that the depth
and width of an $R$-complex can be detected by vanishing of local
(co)homology:
\begin{equation*}
  \depth_RM = \inf\hh^*_\fm(M) \quad\text{and}\quad 
  \width_RM = \inf\hh_*^\fm(M)\,.
\end{equation*}
Combining this with \eqref{eq:lc-Tor} and the isomorphisms
\begin{equation*}
  \RGamma_\fm\LLambda^\fm(M) \simeq \RGamma_\fm(M)
  \quad\text{and} \quad
  \LLambda^\fm\RGamma_\fm(M) \simeq \LLambda^\fm(M)
\end{equation*}
from \cite[Corollary (5.1.1)]{AJL-97} one gets equalities
\begin{align}
  \label{eq:lc-depth}
  \depth_R \RGamma_\fm M &= \depth_R M = \depth_R \LLambda^\fm M \\
  \label{eq:lc-width}
  \width_R \RGamma_\fm M &= \width_R M = \width_R \LLambda^\fm M\,.
\end{align}

For later use, we note that for each $R$-complex $M$ there are
inequalities
\begin{equation}
  \label{eq:lc-dim}
  \sup\hh^{*}_{\fm}(M) \le \dim R + \sup \hh^{*}(M)\quad \text{and}\quad 
  \sup\hh^{*}_{\fm}(M) \le \sup \Ext^*_R(k,M)\,.
\end{equation}
The first is immediate as one has $\hh^*_\fm(M) = \hh^*(C(\fm)
\Ltensor_R M)$ where $C(\fm)$ is the \v{C}ech complex on a system of
parameters for $R$; the second is immediate once one recalls the
isomorphism $\hh^*_\fm(M) \cong \varinjlim_i\Ext^*_R(R/\fm^i,M)$.

The next result is a direct extension of \cite[Proposition~2.1]{AMS90}
by A.-M.\ Simon. Concerning the last assertion:
$\Tor^R_{i}(R,M)=\hh_{i}(M) \ne 0$, so $n$ cannot equal
$\sup\hh_*(M)$, unless both are infinite. However, for later
applications it is convenient to have the statement in this form.

\begin{lemma}
  \label{lem:complete-tor}
  Let $M$ be a derived $\fa$-complete $R$-complex with $\inf\hh_*(M) >
  -\infty$ and $n$ an integer. If $\Tor^R_{n}(R/\fp,M)=0$ all prime
  ideals $\fp \supseteq \fa$, then $\Tor^R_{n}(-,M)=0$.
  
  When in addition $n\ge \sup \hh_{*}(M)$, one has $\flatdim_{R}M\le
  n-1$.
\end{lemma}

\begin{proof}
  First we claim that for any finitely generated $R$-module $L$ and
  integer $i$, if $\fa \Tor^{R}_{i}(L,M)=\Tor^{R}_{i}(L,M)$, then
  $\Tor^{R}_{i}(L,M)=0$. Indeed, let $F$ be a free resolution of $L$
  with each $F_{i}$ finitely generated and equal to zero for $i<0$.
  Let $G$ be a semi-flat resolution of $M$ with $G_{i}=0$ for $i\ll
  0$; this is possible as $\inf\hh_*(M)$ is finite. Since $M$ is
  derived $\fa$-complete, the complex $\varLambda^{\fa}G$, which
  computes $\LLambda^\fa M$, is quasi-isomorphic to $M$. Thus
  \begin{equation*}
    \Tor^{R}_{i}(L,M)=\hh_{i}(F\otimes_{R} \varLambda^{\fa}G)\,.
  \end{equation*}
  Note that $F\otimes_{R}\varLambda^{\fa}G$ is a complex of
  $\fa$-adically complete $R$-modules; this is where we need that each
  $F_{i}$ is finitely generated and that $F_{i}$ and $G_{i}$ are zero
  for $i\ll 0$. It remains to apply \cite[Proposition~1.4]{AMS90}.

  For the stated result, it suffices to prove that the set
  \begin{equation*}
    \{\fb\subset R \text{ an ideal}\mid \Tor^{R}_{n}(R/\fb,M)\ne 0\}
  \end{equation*}
  is empty. Suppose it is not. Pick a maximal element; say, $\fq$. We
  claim that this is a prime ideal. The argument is standard (see, for example, \cite[2.4]{TYLMLR08}) and goes as follows: If it is not, let $\fq'$
  be an associated prime ideal of $R/\fq$, and $x\in R$ an element
  such that $\fq'=\{r\in R\mid xr \in \fq\}$. Then yields an exact
  sequence
  \[
  0\lra R/\fq' \xra{1\mapsto x} R/\fq \lra R/((x)+\fq) \lra 0
  \]
  Since both $\fq'$ and $(x)+\fq$ strictly contain $\fq$, one obtains
  that
  \[
  \Tor^{R}_{n}(R/\fq',M) = 0 = \Tor^{R}_{n}(R/((x)+\fq),M)\,,
  \]
  and hence $\Tor^{R}_{n}(R/\fq,M)=0$, contradicting the choice of
  $\fq$. Thus $\fq$ is prime.

  By hypothesis, $\fq$ does not contain $\fa$, so choose an element
  $a$ in $\fa$ but not in $\fq$ and consider the exact sequence
  \begin{equation*}
    0\lra R/\fq \xra{\ a\ } R/\fq \lra R/((a) + \fq) \lra 0
  \end{equation*}
  Noting that $\Tor^{R}_{n}(R/((a)+\fq),M)=0$ by the choice of $\fq$,
  it follows from the exact sequence above that the map
  $\Tor^{R}_{n}(R/\fq,M)\xra{a} \Tor^{R}_{n}(R/\fq,M)$ is
  surjective. By the claim in the first paragraph, this implies that
  $\Tor^{R}_{n}(R/\fq,M)=0$, which is a contradiction.
\end{proof}

\section{Local rings} 
\label{sec:local}

In this section $(R,\fm,k)$ is a local ring.  Note from
\eqref{eq:lc-dim} that in the next statement $n$ cannot equal
$\sup\hh^*(M)$, but, as with Lemma~\ref{lem:complete-tor}, this
formulation is convenient for later applications.

\begin{lemma}
  \label{lem:bass-finite}
  Let $M$ be a derived $\fm$-torsion $R$-complex with $\inf\hh^*(M) >
  -\infty$. If one has $\Ext^{n}_{R}(k,M)=0$ for some integer $n \ge
  \sup\hh^*(M)$, then $\injdim_R M \le n-1$.
\end{lemma}

\begin{proof}
  Let $I$ be the minimal semi-injective resolution of $M$. One has
  \begin{equation*}
    \Ext^{n}_{R}(k,M) = \hh^{n}(\Hom_{R}(k,I)) 
    = \Hom_{R}(k,I^{n})\,.
  \end{equation*}
  As $M$ is derived $\fm$-torsion, each module $I^i$ is a direct sum
  of copies of the injective envelope of $k$, so $\Ext^{n}_{R}(k,M)=0$
  implies $I^{n}=0$. It follows from the assumption on $n$ and
  minimality of $I$ that $I^i=0$ holds for all $i \ge n$; in
  particular, one has $\injdim_R M \le n-1$.
\end{proof}

The result below extends \eqref{eq:Bass-formula}; its proof would be
significantly shorter under the additional hypothesis that $\inf
\hh^{*}_{\fm}(M)$ is finite.

\begin{proposition}
  \label{prp:bass-finite}
  Let $M$ be an $R$-complex. If\, $\Ext^{n}_{R}(k,M)=0$ holds for some
  integer $n\ge \sup \hh^{*}_{\fm}(M)$, then one has
  $\Ext^{i}_{R}(k,M)=0$ for all $i \ge n$ and
  \begin{equation*}
    \sup\Ext^{*}_{R}(k,M) = \depth R - \width_{R} M\,.
  \end{equation*}
\end{proposition}

\begin{proof}
  Let $J$ be the minimal semi-injective resolution of $M$ and set
  $I:=\varGamma_{\fm}(J)$. For every integer $i$ one has
  \begin{equation*}
    \hh^{i}_{\fm}(M) = \hh^{i}(I) \quad\text{and}\quad \Ext^i_R(k,M) =
    \Hom_R(k,J^i) = \Hom_R(k,I^i)\,.
  \end{equation*}
  Let $d$ be any integer with $n \ge d \ge \sup \hh^{*}_{\fm}(M)$ and
  set $Z := \zz^d(I)$, the submodule of cycles in degree
  $d$. Vanishing of $\hh^{i}_{\fm}(M)$ and the identifications above
  yield
  \begin{equation*}
    \Ext_R^i(k,M) \cong \Ext_R^{i-d}(k,Z)\quad\text{for all $i\ge d\,.$}
  \end{equation*}
  In particular, one has $\Ext_R^{n-d}(k,Z) =0$. Since $Z$ is an
  $\fm$-torsion $R$-module, Lemma \ref{lem:bass-finite} and the
  isomorphisms above yield $\Ext^{i}_{R}(k,M)=0$ for all $i \ge n$.

  It remains to verify the claim about the supremum of
  $\Ext^{*}_{R}(k,M)$. To this end, let $K$ be the Koszul complex on a
  minimal set of generators for $\fm$. It follows from
  \cite[Definition 2.3]{HBFSIn03} that one has the second equivalence
  below
  \begin{alignat*}{2}
    \sup \Ext^{*}_{R}(k,M) = - \infty
    & \iff & &\Ext^{*}_{R}(k,M)=0 \\
    &\iff  & &\hh^{*}(K\otimes_{R}M) = 0 \\
    &\iff & &\width_{R}M = \infty\:.
  \end{alignat*}
  The first one is by definition while the last one is by
  \cite[Theorem 4.1]{HBFSIn03}. We may thus assume that
  $s:=\sup\Ext^*_R(k,M)$ and $w := \sup\hh^*(M \otimes_R K)$ are
  integers. Because $K$ is a bounded complex of finitely generated
  free $R$-modules, there is a quasi-isomorphism
  \begin{equation*}
    \RHom_R(k,M)\Ltensor_R K \simeq \RHom(k,M \otimes_R K)\,.
  \end{equation*}
  From this and the fact that $\Ext^*_R(k,M)$ is a graded $k$-vector
  space, it follows that
  \begin{equation*}
    s = \sup\hh^*(\RHom_R(k,M)\Ltensor_R K)= \sup\Ext^{*}_{R}(k,M\otimes_R K)\,.
  \end{equation*} 
  Let $E$ be the minimal injective resolution of $M\otimes_R K$. Since
  $M\otimes_{R}K$ is derived $\fm$-torsion, one has
  $\varGamma_{\fm}E=E$. From $\Ext^{*}_{R}(k,M\otimes_R K) =
  \Hom_R^*(k,E)$ it thus follows that $E^s \ne 0$ and $E^i=0$ for all
  $i > s$. On the other hand, as $w=\sup \hh^*(E)$ the complex $E^{\ge
    w}$ is the minimal injective resolution of the module $W :=
  \zz^w(E)$ of cycles in degree $w$, so one has $\injdim_R W = s-w$.
  
  It remains to show that $\injdim_{R}W=\depth R$.  Evidently one has
  $\injdim_R W = \sup\Ext^*_R(k,W)$, so by \eqref{eq:Bass-formula} it
  suffices to show that $W$ has width $0$, that is to say, that
  $k\otimes_R W \ne 0$. But this is clear because $\hh^w(E)$ is
  nonzero and annihilated by $\fm$, whence $\fm W \subseteq \bb^w(E)
  \subsetneq W$.
\end{proof}

\begin{proposition}
  \label{prp:betti-finite}
  Let $M$ be an $R$-complex with $\Tor_{n}^{R}(k,M)=0$ for some
  integer~$n$, and assume that one of the following conditions is
  satisfied:
  \begin{enumerate}[\quad\rm(1)]
  \item $n \ge \sup\hh^{\fm}_{*}(M)$ and $\inf\hh_*(M) > -\infty$;
  \item $n\ge \sup \hh^{*}_{\fm}(\Hom_{R}(M,E(k)))$ where $E(k)$ is
    the injective envelope of $k$.
  \end{enumerate}
  One then has $\Tor_i^R(k,M) = 0$ for all $i \ge n$ and
  \begin{equation*}
    \sup\Tor_{*}^{R}(k,M) = \depth R - \depth_{R} M\,.
  \end{equation*}
\end{proposition}

\begin{proof}
  Assume first that (1) is satisfied.  From \eqref{eq:lc-Tor} one gets
  $\Tor^{R}_{n}(k,\LLambda^\fm M)=0$. The complex $\LLambda^{\fm}M$ is
  $\fm$-adically complete and
  \begin{equation*}
    \inf \hh_{*}(\LLambda^{\fm}M)\ge \inf \hh_{*}(M)>-\infty\:.
  \end{equation*}
  Thus Lemma~\ref{lem:complete-tor} applies and yields that
  $\flatdim_R (\LLambda^\fm M)$ is at most $n-1$. Now
  \eqref{eq:lc-Tor} yields $\Tor_i^R(k,M) = 0$ for all $i \ge n$, and
  then \eqref{eq:AB-formula} and \eqref{eq:lc-depth} yield the second
  and third equalities below
  \begin{align*}
    \sup \Tor^{R}_{*}(k,M)
    & = \sup \Tor^{R}_{*}(k,\LLambda^\fm M) \\
    & = \depth R - \depth_{R} (\LLambda^\fm M) \\
    &= \depth R - \depth_{R}M\,.
  \end{align*}

  Assume now that the hypothesis in (2) holds, and set
  $(-)^{\vee}:=\Hom_R(-,E(k))$. For each integer $i$, there is an
  isomorphism
  \begin{equation*}
    \Tor^{R}_{i}(k,M)^{\vee} \cong
    \Ext^{i}_R(k, M^{\vee})\,.
  \end{equation*}
  The hypothesis and Proposition~\ref{prp:bass-finite} yield
  $\Tor_i^R(k,M) = 0$ for all $i \ge n$ and
  \begin{equation*}
    \sup \Tor_*^R(k,M)
    =  \sup \Ext^*_R(k,M^\vee)
    =  \depth R - \width_R M^\vee\,.
  \end{equation*}
  Finally one has $\width_R M^\vee = \depth_R M$; see
  \cite[Proposition 4.4]{HBFSIn03}.
\end{proof}

The final result of this section fleshes out a remark made by Fossum,
Foxby, Griffith, and Reiten at the end of Section 1 in
\cite{FFGR-75}. They phrase it as statement about nonvanishing: If $M$
is an $R$-module and $\Ext_{R}^{n}(k,M)$ is nonzero for some $n \ge
\depth R+1$ then one has $\Ext^{i}_{R}(k,M)\ne 0$ for all $i \ge
n$. The formulation below makes for an easier comparison with
Proposition \ref{prp:bass-finite}.

\begin{proposition}
  \label{prp:bass}
  Let $M$ be an $R$-complex. If $\Ext^{n}_{R}(k,M)= 0$ holds for some
  integer $n\ge \sup \hh^*(M) + \depth R+1$, then one has
  \begin{equation*}
    \Ext^{i}_{R}(k,M)= 0 \quad\text{for every $i$ in the range}
    \ \sup \hh^*(M) + \depth R+1 \le i \le n\:.
  \end{equation*}
\end{proposition}

\begin{proof}
  We may assume that $n>\sup \hh^*(M) + \depth R+1$ holds. It suffices
  to verify that when $\Ext_{R}^{n}(k,M)$ is zero, so is
  $\Ext_{R}^{n-1}(k,M)$.

  Let $\bsx$ be a maximal regular sequence in $R$, set $S:=R/(\bsx)$
  and $\fn:=\fm/(\bsx)$. Thus, $(S,\fn,k)$ is a local ring of depth
  $0$; in particular, $(0:\fn)$, the socle of $S$, is nonzero. Thus,
  there exists a positive integer, say $s$, such that $(0:\fn)$ is
  contained in $\fn^{s}$ but not in $\fn^{s+1}$. Said otherwise, the
  composite of canonical maps
  \begin{equation*}
    (0:\fn)\lra \fn^{s}\lra \fn^{s}/\fn^{s+1}
  \end{equation*}
  is nonzero. Since the source and the target are $k$-vector spaces,
  this implies that $k$ is a direct summand of $\fn^{s}$. It thus
  suffices to verify that $\Ext_{R}^{n-1}(\fn^{s},M)=0$.
  
  The Koszul complex on $\bsx$ is a minimal free resolution of $S$
  over $R$, so one has $\projdim_{R}S=\depth R$ and hence
  \begin{equation*}
    \Ext_{R}^{j}(S,M)=0 \text{ for } j \ge \depth R + 1 + \sup\hh^*(M)\,.
  \end{equation*}
  Given this, the exact sequence
  \begin{equation*}
    0\lra \fn^{s} \lra S\lra S/\fn^{s}\lra 0
  \end{equation*}
  yields an isomorphism
  \begin{equation*}
    \Ext_{R}^{n-1}(\fn^{s}, M) \cong \Ext_{R}^{n}(S/\fn^{s},M)\,.
  \end{equation*}
  Since the length of $S/\fn^{s}$ is finite, $\Ext_{R}^{n}(k,M)=0$
  implies $\Ext_{R}^{n}(S/\fn^{s},M)=0$, and hence the isomorphism
  above yields $\Ext_{R}^{n-1}(\fn^{s}, M)=0$, as desired.
\end{proof}

\section{Flat dimension} 
\label{sec:flatdim}

\noindent
Let $R$ be a commutative noetherian ring, $M$ be an $R$-complex, and
$n$ be an integer. Avramov and Foxby~\cite[Proposition
5.3.F]{LLAHBF91} prove that $\flatdim_{R}M < n$ holds if and only if
one has $\Tor_i^{R_\fp}(k(\fp),M_\fp) =0$ for every prime $\fp$ in $R$
and all $i\ge n$. That is,
\begin{equation}
  \label{eq:flatdim}
  \flatdim_RM = \sup\{i\in\bbZ\mid \Tor_i^{R_\fp}(k(\fp),M_\fp) \ne 0 
  \text{ for some } \fp\in\Spec R\}\:.
\end{equation}
By way of the isomorphisms
\begin{equation}
  \label{eq:tor-isos}
  \Tor_i^{R_\fp}(k(\fp),M_\fp) \cong \Tor_i^R(k(\fp),M)
\end{equation}
this result compares---or may be it is the other way around---to
\cite[Theorem 1.1]{LWCSBI}; see \eqref{eq:injdim}.  Combining
\eqref{eq:flatdim} and \eqref{eq:tor-isos} with \eqref{eq:AB-formula}
one gets
\begin{equation}
  \label{eq:cf}
  \flatdim_R M = \sup_{\fp\in\Spec R}\{\depth R_\fp - \depth_{R_\fp}M_\fp\}
\end{equation}
for every $R$-complex $M$ of finite flat dimension. For modules of
finite flat dimension this equality is known from work of
Chouinard~\cite{LGC76}.

For rings of finite Krull dimension, the next theorem, which contains Theorem~\ref{thm:1},  represents a
significant strengthening of \eqref{eq:flatdim}.

\begin{theorem}
  \label{thm:fdim}
Let $R$ be a commutative noetherian ring and $M$ be an $R$-complex.

If for a prime ideal $\fp$ and $n \ge \dim R_\fp + \sup\hh_*(M)$ one has $\Tor^{R}_{n}(k(\fp),M)=0$, then
\[
\sup\Tor^{R}_{*}(k(\fp),M) = \depth R_\fp - \depth_{R_\fp}M_\fp <  n\,.
\] 
In particular, if there exists an integer $n \ge \dim R + \sup \hh_*(M)$ such that
  \begin{equation*}
    \Tor^{R}_{n}(k(\fp),M)=0 \quad\text{holds for every prime ideal $\fp$ in
      $R$}\:,
  \end{equation*}
then the flat dimension of $M$ is less than $n$.
\end{theorem}

\begin{proof}
It suffices to prove the first claim; the assertion about the flat dimension of $M$ is a consequence, given \eqref{eq:flatdim}.

Fix $\fp$ and $n$ as in the hypotheses. Given \eqref{eq:tor-isos}, this
  yields $\Tor^{R_\fp}_{n}(k(\fp),M_\fp)=0$. One has the following
  (in)equalities
  \begin{equation*}
    \sup\hh_*{(M)} \ge \sup\hh_*{(M_\fp)} =
    \sup\hh^*\Hom_{R_\fp}(M_\fp,E(k(\fp)))\,.
  \end{equation*}
  Since $\dim R \ge \dim R_\fp$, it follows from \eqref{eq:lc-dim} and
  Proposition~\ref{prp:betti-finite} that
  \begin{equation*}
    \sup\Tor^{R_\fp}_{*}(k(\fp),M_\fp) = \depth R_\fp - \depth_{R_\fp}M_\fp < n\:.
  \end{equation*}
This is the desired result.
\end{proof}

The next example shows that the constraint on $n$ in
Theorem~\ref{thm:fdim} is needed.

\begin{example}
  \label{exa:injloc}
  Let $(R,\fm,k)$ be a local ring, $N$ a finitely generated
  Cohen--Macaulay $R$-module of dimension $d$, and set
  $M:=\hh_{\fm}^{d}(N)$.  There are isomorphisms
  \begin{equation*}
    \Tor_{i}^{R}(k,M) \cong \Tor_{i-d}^{R}(k,N) \quad\text{for all
      $i$}\:.
  \end{equation*}
  To see this, let $F$ be the \v{C}ech complex on a maximal
  $N$-regular sequence $\bsx$.  The complex $(\susp^d\, F)\otimes_R N$
  is quasi-isomorphic to $M$, for $\inf\hh^{*}_{\fm}(N)=\sup
  \hh^{*}_{\fm}(N)=d$, by the hypothesis on $N$. Thus, there are
  quasi-isomorphisms
  \begin{equation*}
    k\Ltensor_R M \simeq k\Ltensor_R ((\susp^d\, F)\Ltensor_R N) 
    \simeq \susp^{d}\, k \Ltensor_R N\,.
  \end{equation*}
  The isomorphisms above follow.  Thus one has
  \begin{equation*}
    \Tor_{i}^{R}(k,M) =
    \begin{cases}
      0 & \text{for $i<d$} \\
      N/\fm N & \text{for $i=d\,.$}
    \end{cases}
  \end{equation*}
  In particular, one has $\inf \Tor^{R}_{*}(k,M)=d$, while
  $\Tor^{R}_{*}(k(\fp),M)=0$ for every prime ideal $\fp\ne \fm$, since
  $M$ is $\fm$-torsion.

  Now, if in addition the inequality $d > \depth R$ holds, then
  $\flatdim_{R}M$ is infinite. To see this apply Matlis duality
  $\Tor^R_i(k,M)^{\vee} \cong \Ext_R^i(k,M^{\vee})$ and conclude from
  Proposition~\ref{prp:bass} that $\Tor_i^R(k,M)$ is nonzero for all
  $i\ge d$.

  It remains to remark that such $R$ and $N$ exist: Let $k$ be a
  field, $d$ be a positive integer, and set
  \begin{equation*}
    R:=k[[x_1,\dots, x_{d+1}]]/(x_1^2, x_1x_2,\dots, x_1x_{d+1})\,.
  \end{equation*}
  This $R$ is a local ring of dimension $d$ and depth $0$.  The
  $R$-module $N=R/(x_{1})$ is Cohen--Macaulay of dimension $d$.
\end{example}

\section{Injective dimension} 
\label{sec:injdim}

Let $M$ be an $R$-complex, by \cite[Theorem 1.1]{LWCSBI} one has
\begin{equation}
  \label{eq:injdim}
  \injdim_RM = \sup\{i\in\bbZ\mid \Ext^i_{R}(k(\fp),M) \ne 0 
  \text{ for some } \fp\in\Spec R\}\:.
\end{equation}
In view of \eqref{eq:tor-isos} this is a perfect parallel to the
formula for flat dimension \eqref{eq:flatdim}.

The equality $\flatdim_R M = \sup_{\fp\in\Spec
  R}\{\flatdim_{R_\fp}M_\fp\}$ is immediate from \eqref{eq:flatdim}
and \eqref{eq:tor-isos}. The corresponding equality for the injective
dimension only holds true under extra conditions, and the whole
picture is altogether more complicated.  If $M$ satisfies
$\inf\hh^*(M) > -\infty$, then \cite[Proposition 5.3.I]{LLAHBF91}
yields
\begin{equation}
  \label{eq:injdimbd}
  \begin{aligned}
    \injdim_RM &= \sup\{i\in\bbZ\mid \Ext^i_{R_\fp}(k(\fp),M_\fp) \ne
    0
    \text{ for some } \fp\in\Spec R\} \\
    & = \sup_{\fp\in\Spec R}\{\injdim_{R_\fp}M_\fp\}\:.
  \end{aligned}
\end{equation}
Without the boundedness condition on $\hh(M)$ the injective dimension
may increase under localization; an example is provided in
\ref{exa:injdimloc}.

The next statement, which still requires homological boundedness, is
folklore but not readily available in the literature.

\begin{theorem}
  \label{thm:injdimbd}
  Let $R$ be a commutative noetherian ring and $M$ an $R$-complex with
  $\inf\hh^*(M) > -\infty$. If there exists an integer $n \ge
  \sup\hh^{*}(M)$ such that
  \begin{equation*}
    \Ext^{n+1}_{R_\fp}(k(\fp),M_\fp)=0 \ \ 
    \text{holds for every prime ideal $\fp$ in $R$}\:,
  \end{equation*}
  then the injective dimension of $M$ is at most $n$.
\end{theorem}

\begin{proof}
  Let $I$ be a minimal semi-injective resolution of $M$; as $\inf
  \hh^{*}(M)>-\infty$ holds one has $I^{n}=0$ for $n\ll 0$. For every
  integer $i$ one has $I^i = \coprod_{\fp\in\Spec R}
  E(R/\fp)^{(\mu_i(\fp))}$, and to prove that $\injdim_RM$ is at most
  $n$ it is sufficient to show that the index set $\mu_{n+1}(\fp)$ is
  empty for every prime $\fp$. Fix $\fp$. Since $I_{\fp}$ is a complex
  of injectives with $(I_{\fp})^{n}=0$ for $n\ll 0$, it is a minimal
  semi-injective resolution of $M_\fp$, so one has
  \begin{align*}
    0 &= \Ext^{n+1}_{R_\fp}(k(\fp),M_\fp) \\
    & = \hh^{n+1}\Hom_{R_\fp}(k(\fp),I_\fp) \\
    & = \Hom_{R_\fp}(k(\fp),(I_\fp)^{n+1}) \\
    & = \Hom_{R_\fp}(k(\fp),E(k(\fp))^{(\mu_{n+1}(\fp))})\,.
  \end{align*}
  It follows that $\mu_{n+1}(\fp)$ is empty.
\end{proof}

The next result corresponds to \eqref{eq:cf}. It removes a restriction on the boundedness of $M$
in Yassemi's version \cite[Theorem~2.10]{SYs98b} of
Chouinard's formula \cite[Corollary 3]{LGC76}.

\begin{proposition}
  \label{prp:chouinard}
  For every $R$-complex $M$ of finite injective dimension one has
  \begin{equation*}
    \injdim_R M = \sup_{\fp\in\Spec R}\{\depth R_\fp - \width_{R_\fp}M_\fp\}\:.
  \end{equation*}
\end{proposition}

\begin{proof}
  Without loss of generality, we can assume that $M$ is semi-injective
  with $M^i=0$ for all $i > d := \injdim_RM$. For every $u\le d$ there
  is an exact sequence of semi-injective complexes
  \begin{equation*}
    \tag{1}
    0 \lra M^{\ge u} \lra M \lra M^{\le u-1} \lra 0
  \end{equation*}
  with $\injdim_R M^{\le u-1} \le u-1$ and $\injdim_R M^{\ge u} =
  d$. The complex $M^{\ge u}$ is bounded, so \eqref{eq:Bass-formula}
  and \eqref{eq:injdimbd} conspire to yield
  \begin{equation*}
    \tag{2}
    d = \sup_{\fp\in\Spec R}\{\depth R_\fp - \width_{R_\fp}M^{\ge u}_\fp\}\:.
  \end{equation*}

  First we establish the inequality
  \begin{equation*}
    \tag{3}
    \injdim_R M \ge \sup_{\fp\in\Spec R}\{\depth R_\fp - \width_{R_\fp}M_\fp\}\:.
  \end{equation*}
  Let $\fp$ be a prime. There are inequalities
  \begin{equation*}
    \width_{R_\fp} M_\fp \ge 
    \inf\hh_*(M_\fp) \ge \inf\hh_*(M) = - \sup\hh^*(M) \ge -d\,,
  \end{equation*}
  and without loss of generality one can assume that $\width_{R_\fp}
  M_\fp$ is finite. Consider $(1)$ for $u = -\width_{R_\fp} M_\fp$ and
  localize at $\fp$. The associated exact sequence of Tor groups
  yields $\width_{R_\fp} M_\fp = \width_{R_\fp} M^{\ge u}_\fp$, so the
  desired inequality $d \ge \depth R_\fp - \width_{R_\fp}M_\fp$
  follows from (2). It remains to prove that equality holds for some
  prime.

  Consider (1) for $u=d-1$ and choose by (2) a prime $\fp$ with
  \begin{equation*}
    d = \depth R_\fp - \width_{R_\fp}M^{\ge d-1}_\fp\:.
  \end{equation*}
  The second inequality in the next display is (3) applied to the
  complex $M^{\le d-2}$.
  \begin{equation*}
    d-2 \ge \injdim_RM^{\le d-2} \ge \depth R_\fp - \width_{R_\fp} M^{\le
      d-2}_\fp\:.
  \end{equation*}
  Eliminating $d$ and $\depth R_\fp$ between the two displays one gets
  the inequality
  \begin{equation*}
    \width_{R_\fp}M^{\ge d-1}_\fp \le \width_{R_\fp} M^{\le d-2}_\fp
    -2\:.
  \end{equation*}
  Finally, one gets $\width_{R_\fp}M_\fp = \width_{R_\fp}M^{\ge
    d-1}_\fp$ from the exact sequence of Tor groups associated to (1).
\end{proof}

We now aim for a characterization of complexes of finite injective
dimension that does not require homological boundedness.  It is based
on the following observation, of independent interest.

\begin{lemma}
  \label{lem:maximal}
  Let $M$ be an $R$-complex and $\fm$ a maximal ideal in $R$. The
  localization maps $\rho \colon M\to M_{\fm}$ and $\sigma\colon R\to
  R_{\fm}$ induce quasi-isomorphisms
  \begin{align*}
    \RHom_{R}(k(\fm),\rho) &\colon
    \RHom_R(k(\fm),M) \xra{\ \simeq\ } \RHom_R(k(\fm),M_\fm) \quad\text{and}\\
    k(\fm)\Ltensor_{R} \RHom_{R}(\sigma, M) &\colon
    k(\fm)\Ltensor_{R}\RHom_{R}(R_{\fm},M) \xra{\;\,\simeq\;\,}
    k(\fm)\Ltensor_{R} M\,.
  \end{align*}
\end{lemma}

\begin{proof}
  In the derived category of $R$, consider the distinguished triangle
  \begin{equation*}
    M \xra{\;\,\rho\;\,} M_\fm \lra C \lra \susp M\,.
  \end{equation*}
  The induced morphism $k(\fm)\Ltensor_R \rho$ is a quasi-isomorphism,
  so $k(\fm)\Ltensor_R C$ is acyclic. Then $\RHom_R(k(\fm),C)$ is also
  acyclic, by \cite[Theorem 4.13]{BIK-12}, whence the map
  $\RHom_{R}(k(\fm),\rho)$ is a quasi-isomorphism, as claimed.

  In the same vein, the distinguished triangle $R \to R_\fm \to D \to
  \susp R$ induces a distinguished triangle
  \begin{equation*}
    \RHom_{R}(D,M) \lra \RHom_{R}(R_{\fm},M)
    \xra{\;\,\RHom_{R}(\sigma,M)\;\,} M \lra \susp \RHom_{R}(D,M)\,.
  \end{equation*}
  By adjunction $\RHom_{R}(k(\fm),\RHom_{R}(\sigma,M))$ is a
  quasi-isomorphism, so that
  \begin{equation*}
    \RHom_{R}(k(\fm),\RHom_{R}(D,M))
  \end{equation*}
  is acyclic. Thus the complex $k(\fm)\Ltensor_R\RHom_R(D,M)$ is
  acyclic by \cite[Theorem 4.13]{BIK-12}, which justifies the second
  of the desired quasi-isomorphisms.
\end{proof}

\begin{proposition}
  \label{prp:artin}
  Let $M$ be an $R$-complex and $\fm$ a maximal ideal in $R$. If one
  has $\Ext_R^n(k(\fm),M)=0$ for some $n\ge \dim R_\fm +
  \sup\hh^*(M_\fm)$, then $\Ext_R^i(k(\fm),M)=0$ holds for all $i \ge
  n$ and there are equalities
  \begin{align*}
    \sup\Ext^{*}_{R}(k(\fm),M) &= \depth R_\fm - \width_{R_\fm} M_\fm \\
    &= \depth R_\fm - \width_{R_\fm} \RHom_R(R_\fm,M)\,.
  \end{align*}
\end{proposition}

\begin{proof}
  The first isomorphism below is by Lemma~\ref{lem:maximal}; the
  second is by adjunction.
  \begin{equation*}
    \Ext_R^*(k(\fm),M) \cong \Ext_R^*(k(\fm),M_\fm) 
    \cong \Ext^{*}_{R_\fm}(k(\fm),M_\fm)\,.
  \end{equation*}
  In view of these isomorphisms and the assumption on $n$, Proposition
  \ref{prp:bass-finite} now yields
  \begin{equation*}
    \sup\Ext^{*}_{R}(k(\fm),M) = \depth R_\fm - \width_{R_\fm} M_\fm < n\,.
  \end{equation*}
  It remains to observe that the width of $M_\fm$ and
  $\RHom_R(R_\fm,M)$ coincide, by the second quasi-isomorphism in
  Lemma~\ref{lem:maximal}.
\end{proof}

\begin{corollary}
  \label{co:injdim}
  Let $R$ be an artinian ring and $M$ an $R$-complex. If there exists
  an integer $n \ge \sup\hh^{*}(M)$ such that
  \begin{equation*}
    \Ext^{n}_{R}(k(\fp),M)=0 \quad
    \text{holds for every prime ideal $\fp$ in $R$}\:,
  \end{equation*}
  then the injective dimension of $M$ is less than $n$.
\end{corollary}

\begin{proof}
  Every prime ideal $\fp$ in $R$ is maximal and there are inequalities
  \begin{equation*}
    n \ge \sup\hh^*(M) \ge \sup\hh^*(M_\fp) = \dim R_\fp + \sup\hh^*(M_\fp)\,.
  \end{equation*}
  Thus the claim follows from \eqref{eq:injdim} and Proposition
  \ref{prp:artin}.
\end{proof}

\begin{remark}
  \label{rem:splf}
  In the sequel we require the invariant
  \begin{equation*}
    \splf R = \sup\{\projdim_RF \mid F\ \text{\rm is a flat
      $R$-module}\}.
  \end{equation*}
  Every flat $R$-module is projective if and only if $R$ is artinian,
  so $\splf R > 0$ holds if $\dim R > 0$, and from work of Jensen
  \cite[Proposition~6]{CUJ70} and Raynaud and Gruson
  \cite[thm.~II.3.2.6]{LGrMRn71} one gets the upper bound $\splf R \le
  \dim R$. A result of Gruson and Jensen \cite[Theorem~7.10]{LGrCUJ81}
  yields another bound on the invariant $\splf R$: If $R$ has
  cardinality at most $\aleph_m$ for some natural number $m$, then one
  has $\splf R \le m+1$. Thus for countable rings, and for
  $1$-dimensional rings, the bound in the theorem below is $n \ge \dim
  R + \sup\hh^*(M)$.
\end{remark}

\begin{theorem}
  \label{thm:injdim}
  Let $R$ be a commutative noetherian ring with $\dim R \ge 1$ and $M$
  an $R$-complex. If there exists an integer $n \ge \dim R - 1 + \splf
  R + \sup\hh^{*}(M)$ such that
  \begin{equation*}
    \Ext^{n}_{R}(k(\fp),M)=0 \quad
    \text{holds for every prime ideal $\fp$ in $R$}\:,
  \end{equation*}
  then the injective dimension of $M$ is less than $n$.
\end{theorem}

\begin{proof}
  Fix a prime ideal $\fp$ and consider the $R_{\fp}$-complex
  $N:=\RHom_{R}(R_{\fp},M)$.  One has
  \begin{equation*}
    \hh^{i}(N) = \Ext_{R}^{i}(R_{\fp},M) = 0 
    \quad\text{for } i> \projdim_RR_\fp + \sup\hh^{*}(M)\,,
  \end{equation*} 
  and standard adjunction yields
  \begin{equation*}
    \Ext_{R}^{*}(k(\fp),M)\cong \Ext_{R_{\fp}}^{*}(k(\fp),N)\,.
  \end{equation*}
  Since $R_{\fp}$ is a flat $R$-module, $\projdim_RR_\fp$ is at most
  $\splf R$.  Given this, \eqref{eq:lc-dim} yields
  \begin{equation*}
    \hh_{\fp R_{\fp}}^{i}(N)=0 \quad\text{for } 
    i > \dim R_{\fp} + \splf R + \sup\hh^{*}(M)\,.
  \end{equation*}
  Thus, if $n \ge \dim R_{\fp} + \splf R + \sup\hh^{*}(M)$ holds, then
  Proposition~\ref{prp:bass-finite} yields
  \begin{equation*}
    \sup\Ext^{*}_{R}(k(\fp),M)= \depth R_\fp - \width_{R_\fp} N < n\,.
  \end{equation*}
  The same equality also holds when $n < \dim R_{\fp} + \splf R +
  \sup\hh^{*}(M)$, for then the assumption on $n$ forces $\dim R_\fp =
  \dim R$, so $\fp$ is a maximal ideal and so
  Proposition~\ref{prp:artin} applies.  Now the desired conclusions
  follows from \eqref{eq:injdim}.
\end{proof}

\begin{remark}
  \label{rmk:injdim}
  As noted in Remark~\ref{rem:splf}, one has $\splf R\leq \dim
  R$. Thus Theorem~\ref{thm:3} is a consequence of
  Corollary~\ref{co:injdim} and Theorem~\ref{thm:injdim}.
\end{remark}

Confer the following result and Proposition~\ref{prp:chouinard}.

\begin{corollary}
  \label{cor:chouinard}
  For every $R$-complex of finite injective dimension one has
  \begin{equation*}
    \injdim_R M = \sup_{\fp\in\Spec R}\{\depth R_\fp -
    \width_{R_\fp}\RHom_R(R_\fp,M)\}\,.
  \end{equation*}
\end{corollary}

\begin{proof}
  Given Lemma~\ref{lem:maximal}, the desired equality is restatement
  of Proposition \ref{prp:chouinard} in case $R$ is artinian.  If $R$
  is not artinian, then one has $\dim R \ge 1$ and the equality is
  immediate from \eqref{eq:injdim} and the last display in the proof
  of Theorem \ref{thm:injdim}.
\end{proof}

\begin{remark}
  Let $R$ be a complete local domain of positive dimension. One has
  $\width_{R_{(0)}}R_{(0)} =0$, but the complex $\RHom_R(R_{(0)},R)$
  is acyclic, see \cite[Example 4.20]{BIK-12}, so
  $\width_{R_{(0)}}\RHom_R(R_{(0)},R) = \infty$. We do not
  know how the numbers $\width_{R_\fp}M_\fp$ and
  $\width_{R_\fp}\RHom_R(R_\fp,M)\}$ from
  Proposition~\ref{prp:chouinard} and Corollary~\ref{cor:chouinard} 
  compare in general.
\end{remark}

\section{Examples} 
\label{sec:exa}

In this section we describe examples to illustrate that, for complexes
whose cohomology is not bounded below, finiteness of injective
dimension does not behave well under localization or passage to
torsion subcomplexes. This builds on \cite{XWCSBI10}.

\begin{remark}
  \label{rem:semiinjective}
  Let $R$ be a ring. A complex $I$ of injective $R$-modules is
  semi-injective if and only if for each (equivalently, for some)
  integer $n$ the quotient complex $I^{\le n}$ is semi-injective. This
  is immediate from the exact sequence of complexes
  \begin{equation*}
    0\lra I^{> n} \lra I \lra I^{\le n}\lra 0
  \end{equation*}
  since $I^{>n}$ is always semi-injective.
\end{remark}

\begin{remark}
  \label{rem:faithful}
  Let $(R,\fm,k)$ be a local ring and $E$ be the injective envelope of
  $k$. One has $(E^\bbN)_\fp\ne 0$ for every prime ideal in
  $R$. Indeed, the claim is trivial if $R$ is artinian. If $R$ is not
  artinian, then one can choose an element $e = (e_n)_{n\in\bbN}$ in
  $E^\bbN$ with $\fm^n \subseteq (0:e_n) \not\supseteq \fm^{n-1}$. The
  map $R \to E^\bbN$ given by $1 \mapsto e$ is injective by Krull's
  intersection theorem, so $R$ is a submodule of $E^\bbN$.
\end{remark}

\begin{example}
  \label{exa:injdimloc}
  Let $(R,\fm,k)$ be a local ring such that $(0:x) = (x)$ holds for
  some $x\in \fm$; set $S = R/(x)$. The complex
  \begin{equation*}
    \cdots \xra{\:\:x\:\:} R \xra{\:\:x\:\:} R \xra{\:\:x\:\:} R \lra 0
  \end{equation*}
  concentrated in nonnegative degrees has homology $S$ in degree $0$
  and zero elsewhere. Dualizing with respect to $E$, the injective
  envelope of $k$ over $R$, yields a complex
  \begin{equation*}
    I := 0 \lra E \xra{\:\:x\:\:} E \xra{\:\:x\:\:} E \xra{\:\:x\:\:} \cdots
  \end{equation*}
  of injective $R$-modules. It is the minimal injective resolution of
  $E_S := \Hom_R(S,E)$ over $R$. By periodicity, every injective
  syzygy of $E_S$ is $E_S$. Consider the complex $J =
  \prod_{n>0}\susp^n I$, which is a semi-injective resolution of
  $\prod_{n>0}\susp^n E_S$.

  \medskip

  \emph{Claim}.\ The complex $M := J^{\le 0}$ has injective dimension
  $0$, whereas for each prime ideal $\fp \ne \fm$, one has that
  $\injdim_{R_\fp}M_\fp$ is infinite.

  \medskip

  Indeed, since $J$ is semi-injective, so is $M$, by
  Remark~\ref{rem:semiinjective}. Since the cohomology module
  $\hh^0(M) \cong (E_S)^\bbN$ is nonzero, it follows that
  $\injdim_{R}M=0$ holds.

  Fix a prime ideal $\fp \ne \fm$. For $i <0$ one has $\hh^i(M) =
  \hh^i(J) = E_S$ and, therefore, $\hh^i(M_\fp) =
  \hh^{i}(M)_{\fp}=0$. This justifies the first quasi-isomorphism in
  the computation below; the rest are standard.
  \begin{align*}
    \RHom_{R_\fp}(k(\fp),M_\fp) & \simeq
    \RHom_{R_\fp}(k(\fp), ((E_S)^\bbN)_\fp) \\
    &\simeq \RHom_{R}(R/\fp, (E_S)^\bbN)_\fp \\
    &\simeq (\RHom_{R}(R/\fp,  E_S)^\bbN)_\fp \\
    &\simeq (\Hom_{R}(R/\fp,  I)^\bbN)_\fp \\
    &\cong \Big(\big(\coprod_{i\ge 0}\susp^{-i}
    \Hom_{R}(R/\fp,E)\big)^\bbN\Big)_\fp \\
    &\cong \coprod_{i\ge0}\susp^{-i} (\Hom_{R}(R/\fp, E)^\bbN)_\fp\:.
  \end{align*}
  The first isomorphism holds because $x=0$ in $R/\fp$, for $x^{2}=0$,
  and hence the induced differential on the complex $\Hom_{R}(R/\fp,
  I)$ is zero.

  The module $E_{R/\fp} := \Hom_R(R/\fp,E)$ is the injective envelope
  of $k$ over the domain $R/\fp$. The computation above shows that for
  every $i\ge 0$ there is an isomorphism as $R/\fp$-modules
  \begin{equation*}
    \Ext^i_{R_\fp}(k(\fp),M_\fp) \cong ((E_{R/\fp})^\bbN)_{(0)}\:.
  \end{equation*}
  Thus Remark~\ref{rem:faithful} yields
  $\Ext^i_{R_\fp}(k(\fp),M_\fp)\ne 0$ for all $i\ge0$; hence
  $\injdim_{R_\fp}M_\fp$ is infinite, as claimed.
\end{example}

\begin{example}
  Let $k$ be a field and $R:=k[|x,y|]/(x^{2})$. Since $(0:x) = (x)$,
  we are in the situation considered in the previous example. Let $M$
  be the complex of injectives with injective dimension zero
  constructed there. We claim that the injective dimension of the
  complexes $M_{y}$ and $\varGamma_{(y)}M$ are infinite.

  Indeed, observe that $M_{y}\cong M_{\fp}$ where $\fp$ is the prime
  ideal $(x)$ of $R$, so $\injdim_{R}M_{y}$ is infinite, by the claim
  in the previous example. Since $C(y)\otimes_R M$ is quasi-isomorphic
  to $\varGamma_{(y)}M$ and there is an exact sequence
  \begin{equation*}
    0\lra \susp^{-1}M_{y} \lra C(y)\otimes_R M \lra M \lra 0,
  \end{equation*}
  it follows that the injective dimension of $\varGamma_{(y)}M$ is
  infinite as well.
\end{example}

\section*{Acknowledgments} 

We thank Peder Thompson for conversations that helped clarify the
material in Section~\ref{sec:injdim}.

\providecommand{\MR}[1]{\mbox{\href{http://www.ams.org/mathscinet-getitem?mr=#1}{#1}}}
\renewcommand{\MR}[1]{\mbox{\href{http://www.ams.org/mathscinet-getitem?mr=#1}{#1}}}
\providecommand{\arxiv}[2][AC]{\mbox{\href{http://arxiv.org/abs/#2}{\sf
      arXiv:#2 [math.#1]}}} \def\cprime{$'$}
\providecommand{\MR}{\relax\ifhmode\unskip\space\fi MR }
\providecommand{\MRhref}[2]{%
  \href{http://www.ams.org/mathscinet-getitem?mr=#1}{#2} }
\providecommand{\href}[2]{#2}

\end{document}